\documentclass[11pt]{amsart}

\usepackage{amssymb,amsmath,amsthm,amscd,mathrsfs,graphicx}
\usepackage[cmtip,all]{xy}
\usepackage{pb-diagram}
\usepackage{tikz}

\numberwithin{equation}{section}

\newtheorem{theorem}{Theorem}[section]

\newtheorem{lemma}{Lemma}[section]

\newtheorem{question}{Question}[section]

\newcommand{\ov}[1]{\overline{#1}}

\theoremstyle{definition}

\theoremstyle{remark}

\begin{document}
\bibliographystyle{amsplain}

\title[Instantaneous convexity breaking]{Instantaneous convexity breaking for the quasi-static droplet model}

\author[A. Chau]{Albert Chau}
\address{Department of Mathematics, The University of British Columbia, 1984 Mathematics Road, Vancouver, B.C.,  Canada V6T 1Z2.  Email: chau@math.ubc.ca. } 
\author[B. Weinkove]{Ben Weinkove}
\address{Department of Mathematics, Northwestern University, 2033 Sheridan Road, Evanston, IL 60208, USA.  Email: weinkove@math.northwestern.edu.}

\thanks{Research supported in part by  NSERC grant $\#$327637-06 and NSF grant DMS-2005311.}

\maketitle

\vspace{-20pt}

\begin{abstract}  We consider a well-known quasi-static model for  the shape of a liquid droplet.  The solution can be described in terms of  time-evolving domains in $\mathbb{R}^n$.
We give an example to show that convexity of the domain can be instantaneously broken.
\end{abstract}

\maketitle
\section{Introduction}

We consider the following system of equations for a function $u(x,t)$ and domains $\Omega_t \subset \mathbb{R}^n$, for $t \ge 0$.  This system is used to model the quasi-static shape evolution of a liquid droplet of height $u(x,t)$ occupying the region $\Omega_t$:
\begin{equation} \label{dm}
\begin{split}
 - \Delta u = {} & \lambda_t, \quad \textrm{on } \Omega_t \\
 u = {} & 0, \quad \textrm{on } \partial \Omega_t \\
 V = {} & F(|Du|), \quad \textrm{on } \partial \Omega_t \\
 \int_{\Omega_t} u \, dx = {} & 1.
\end{split}
\end{equation}
In the above, $V$ is the velocity of the free boundary $\partial \Omega_t$ in the direction of the outward unit normal and $F: (0,\infty) \rightarrow \mathbb{R}$ is an analytic function with $F'(r)>0$ for $r>0$.  The constant $\lambda_t>0$ is determined by the integral condition on $u$.

The initial data is given by a domain $\Omega_0$ which we assume is bounded with smooth boundary $\partial \Omega_0$.   Note that the domains $\Omega_t$ (assuming they are bounded with sufficiently regular boundary $\partial \Omega_t$) determine uniquely the solution $x \mapsto u(x,t)$.  Thus we may denote a solution of (\ref{dm}) by a family of evolving domains $\Omega_t$.
 In Section \ref{sectionshort} we will explain what is meant by a \emph{classical solution} to this problem.

The system of equations (\ref{dm}) has long been accepted as a model for droplet evolution in the physical literature \cite{C, Gl, Gr, T, V}.  There have been results on weak formulations  of this equation by Glasner-Kim \cite{GlK} and Grunewald-Kim \cite{GK}.  Feldman-Kim  \cite{FK} gave some conditions for global existence and convergence to an equilibrium.  Escher-Guidotti \cite{EG} proved a short time existence result for classical solutions, which we describe in Section \ref{sectionshort} below.

In this note we address the following natural question:

\begin{question} \label{question}
Is the convexity of $\Omega_t$  preserved by the system (\ref{dm})?
\end{question}

This question is implicit in the work of Glasner-Kim \cite{GlK}.  It  was raised explicitly by Feldman-Kim \cite[p.822]{FK}, ``Let us point out that, in particular, it is unknown whether the convexity of the drop is preserved in the system [(\ref{dm})].''

 In this note, we answer Question \ref{question} by showing that convexity is \emph{not} generally preserved.  We make an assumption on $F$, namely that
 \begin{equation} \label{assumption}
 \lim_{r \rightarrow 0^+} \frac{F''(r)}{F'(r)} \ge \gamma, \quad \textrm{for some } \gamma>0.
 \end{equation} 

This includes the important cases $F(r) = r^3-1$ and $F(r)=r^2-1$ considered in \cite{GlK} and \cite{FK, GK} respectively.  

We construct an example where $\Omega_t$ is convex for $t=0$, but not convex for $t\in (0, \delta)$ for some $\delta>0$.
 
 \begin{theorem}\label{mainthm}  Assume $F$ satisfies assumption (\ref{assumption}).  
 	There exists $\delta>0$ and a bounded convex domain $\Omega_0 \subset \mathbb{R}^2$ with smooth boundary such that the solution $\Omega_t$ to \eqref{dm} with this initial data is not convex for any $t\in (0, \delta]$.
 \end{theorem}
 
 Escher-Guidotti \cite{EG} showed that as long as $\Omega_0$ is a bounded domain with sufficiently smooth boundary, there always exists a unique classical solution for a short time, and this is what is meant by ``the solution $\Omega_t$'' in the statement of Theorem \ref{mainthm}.  In Section \ref{sectionshort}, we describe more precisely the results of \cite{EG}.

In Section \ref{proof} we give the proof of Theorem \ref{mainthm}.  The starting point is an  explicit solution of the equation $-\Delta u =\lambda_0$ on an equilateral triangle \cite{KM}.  We smooth out the corners to obtain our convex domain $\Omega_0$, and  show that it immediately breaks convexity.

\section{Short time existence} \label{sectionshort}

In this section, we recall the short time existence result of Escher-Guidotti \cite{EG}.

We first give a definition of a solution of (\ref{dm}), following \cite{EG}.  Note that the domains $\Omega_t$ determine uniquely the functions $u$, so we will describe the solution of (\ref{dm}) in terms of varying domains - given as graphs over the original boundary.

Fix $\alpha \in (0,1)$.  Assume $\Omega_0$ is a bounded domain in $\mathbb{R}^n$ whose boundary $\Gamma_0:=\partial \Omega_0$ is a smooth hypersurface. Let $\nu(x)$ denote the unit outward normal to $\Gamma_0$ at $x$.  Then there exists a maximal constant $\sigma(\Omega_0)>0$ such that for any given function $\rho \in C^{2+\alpha}(\Gamma_0)$ with $\| \rho \|_{C^1(\Gamma_0)}\le \sigma$, the set
$$\Gamma_{\rho} = \{ x + \rho(x) \nu(x) \ | \ x \in \Gamma_0 \},$$
is a $C^{2+{\alpha}}$ hypersurface in $\mathbb{R}^n$ which is the boundary of a bounded domain $\Omega=\Omega(\rho)$.

We can now describe a solution of (\ref{dm}) in terms of a time-varying family $\rho(x,t)$.  Namely, given
$$\rho \in C([0,T], C^{2+\alpha}(\Gamma_0)) \cap C^1([0,T], C^{1+\alpha}(\Gamma_0)),$$
with $\sup_{t \in [0,T]}\| \rho(\cdot, t) \|_{C^1(\Gamma_0)}< \sigma(\Omega_0)$, 
write $\Omega_t$, for $t \in [0,T]$ for the corresponding family of domains, with boundaries $\Gamma_t:=\Gamma_{\rho(t)}$.  The velocity $V$ of the boundary in the direction of the outward normal, at a point $y=x+\rho(x,t) \nu(x) \in \Gamma_t$ is given by
$$V = \frac{\partial \rho}{\partial t} (x,t) \nu(x) \cdot n(y,t),$$
where $n(y,t)$ is the outward unit normal to $\Gamma_t$ at the point $y$.

Since the domains $\Omega_t$ have $C^{2+\alpha}$ boundaries, there exists for each $t$ a unique solution $u(\cdot, t) \in C^{2+\alpha}(\ov{\Omega}_t)$ and $\lambda_t \in \mathbb{R}$ of 
$$-\Delta u = \lambda_t, \quad \textrm{on } \Omega_t, \quad u|_{\Gamma_t} =0, \quad \int_{\Omega_t}u\, dx =1,$$
(see for example \cite[Theorem 6.14]{GT}).

Then we say that such a $\rho$ is a \emph{classical solution} of (\ref{dm}) with initial domain $\Omega_0$ if the velocity $V(y)$ at each $y \in \Gamma_t$, for $t \in [0,T]$ satisfies
$$V = F(|Du|).$$
 
 The main theorem of Escher-Guidotti \cite{EG} implies in particular the following:
 
 \begin{theorem} \label{thmEG}
 There exists a $T>0$ and a unique classical solution 
$$\rho \in C([0,T], C^{2+\alpha}(\Gamma_0)) \cap C^1([0,T], C^{1+\alpha}(\Gamma_0))$$
 of the quasi-static droplet model (\ref{dm}) with initial domain $\Omega_0$ whose  boundary $\Gamma_0$ is smooth.
 \end{theorem}
 
 In fact they prove more:  they also allow their initial domain to have boundary in $C^{2+\alpha}$.  Note that this result does not require the assumption (\ref{assumption}).

  \section{Proof of Theorem \ref{mainthm}} \label{proof}
  
 In this section we give a proof of Theorem \ref{mainthm}.  We work in $\mathbb{R}^2$, using $x$ and $y$ as coordinates. The heart of the proof is the following lemma, which makes use of the assumption (\ref{assumption}).
   
 \begin{lemma}\label{initialdata}
 	There exists a bounded convex domain $\Omega_0$ with smooth boundary $\Gamma_0$, and real numbers $0<x_0<x_1$ with the following properties:  
 	\begin{enumerate}
 		\item[(i)] $\Omega_0$ is contained in $\{ y\ge 0\}$.
 		\item[(ii)] $(x, 0)\in \partial \Omega_0$ for $x_0 \le x \le x_1$.
 		\item[(iii)] Let $u(x,y)$ solve
	$$-\Delta u = \lambda_0, \quad \textrm{on } \Omega_0, \quad u|_{\Gamma_0} =0, \quad \int_{\Omega_0}u\, dxdy =1,$$
	for a constant $\lambda_0$.
Then $V(x):=F(|Du(x, 0)|)$ satisfies $$\frac{V(x_0)+V(x_1)}{2}> V\left( \frac{x_0+x_1}{2} \right).$$
	\end{enumerate}
 	\end{lemma} 
  
  \begin{proof}
  	We begin with the following explicit solution of the ``torsion problem,'' $-\Delta v = \textrm{const}$, on the equilateral triangle  \cite{KM}.  Let $D$ be the equilateral triangle of side length $2a$ given by
  	$$y>0, \ \sqrt{3} |x| > y-a\sqrt{3}.$$
  	The function 
  	$$v = c y ((y-a\sqrt{3})^2 - 3x^2), \quad \textrm{for } c:= \frac{5}{3a^5},$$
  	satisfies $$-\Delta v =4ac\sqrt{3},$$ vanishes on the boundary of $D$ and satisfies
	$$\int_{D} v \, dxdy=1.$$

On the bottom edge of the triangle 
  	$$E_1 = \{ (x, 0) \in \mathbb{R}^2 \ | \ -a \le x \le a \},$$
  	we have
  	$$v_y(x, 0) =  3c (a^2-x^2).$$
  	Hence
  	$$V(x): = F(3c(a^2-x^2)),$$
  and 
  \begin{equation}\label{uno}
  V''(x) = 36 c^2 x^2 F''(3c(a^2-x^2)) - 6cF'(3c(a^2-x^2)).
  \end{equation}
  Recalling that $c=5/(3a^5)$, then we may choose $a>0$ sufficiently small so that
  \begin{equation} \label{dos}
  36c^2 x^2 \ge 2\frac{6c}{\gamma}, \quad \textrm{for }|x| \ge a/2,
  \end{equation}
 where $\gamma>0$ is given by our assumption (\ref{assumption}).    From now on we fix this $a$ (and hence $c$).
  
   It follows from (\ref{uno}), (\ref{dos}) and (\ref{assumption}) that $V''(x)>0$ for $|x|$ sufficiently close to $a$.  In particular there exists $0<x_0<x_1<a$ with
\begin{equation}\label{Veqn}
\frac{V(x_0)+V(x_1)}{2}> V\left( \frac{x_0+x_1}{2} \right).
\end{equation}

  	The above example above readily implies the existence of a smooth domain $\Omega_0$  satisfying the conditions in the Lemma. 	Indeed, we only have to ``smooth the corners'' of the triangle domain $D$. 
	
Denote the vertices of $D$ by $p_1, p_2, p_3$.  Let $ \{D_k\}_{k=1}^{\infty}$ be a sequence of bounded convex domains with smooth boundaries such that  for each $k\geq 1$:
  	\begin{enumerate}
  		\item $D_k\subset D_{k+1}\subset D$ (the sequence is nested and increasing).
  		\item $D \setminus D_k  \subset \bigcup_{i=1}^3 B_{k^{-1}}(p_i)$, where $B_r(p)$ denotes the ball of radius $r$ centered at $p$. 
  	\end{enumerate}
Such a sequence $\{ D_k \}$ can be constructed by ``rounding out the corners'' of the triangle $D$ in a ball of radius $k^{-1}$ centered at each corner.

For each $k\geq 1$ let $u_k$ on $D_k$ be the solutions of 
$$-\Delta u_k = 4ac\sqrt{3}, \quad \textrm{on } D_k, \quad u|_{\partial D_k} =0,$$
where we recall that $a$ and $c$ are fixed constants.

It follows from property (1) above and the maximum principle that for each $k\geq 1$
   \begin{equation}\label{decreasing}
   	0<u_{k}\leq u_{k+1}\leq v, \quad \textrm{on } D_k,
   	\end{equation}
  	from which we conclude a pointwise limit on the triangle $D$
  	\begin{equation}
  		0\le u_{\infty}(x):= \lim_{k\to \infty} u_k(x) \leq v(x), \quad \textrm{for } x\in D,
  	\end{equation}
  	and define $u_{\infty}(x)$ to be zero on $\partial D$.
	
  	By standard elliptic estimates (see for example \cite[Theorem 6.19]{GT} and the remark after it), the convergence above will hold in $C^{\ell} (K)$ for any compact set $K\subset \subset (\ov{D}\setminus \{p_1, p_2, p_3\})$ and any $\ell \ge 0$.  Hence $u_{\infty} \in C^{\infty}(\ov{D}\setminus \{p_1, p_2, p_3\})$ and $-\Delta u_{\infty} =4ac\sqrt{3}$ on $D$.  Moreover, by \eqref{decreasing} and the continuity of $v$ it is easily verified that $u_{\infty}$ is also continuous at the corners $p_1, p_2, p_3$ and thus on all of $\ov{D}$.   By the maximum principle, $u_{\infty}=v$.  Note also that 
	$$\int_{D_k} u_k\, dxdy \rightarrow 1, \quad \textrm{as } k \rightarrow \infty.$$
  	
Then for sufficiently large $k$ the domain $\Omega_0:= D_k$ will satisfy conditions (i), (ii), (iii), with $$u: =  \frac{u_k}{\int_{D_k} u_k\, dxdy}, \quad \lambda_0 := \frac{4ac\sqrt{3}}{\int_{D_k} u_k\, dxdy}.$$
Here we are using (\ref{Veqn}) and the fact  that $x\mapsto F(|Du_k(x,0)|)$ will converge uniformly to $x\mapsto F(|Dv(x,0)|)$ on $[x_0, x_1]$ as $k \rightarrow \infty$.
This completes the proof of the lemma.
    	\end{proof}

\begin{proof}[Proof of Theorem \ref{mainthm}]  Let $\Omega_0$ and $u$ be given as in Lemma \ref{initialdata}. By Theorem \ref{thmEG}, there exists a unique classical solution of (\ref{dm}) for a short time interval $[0,T]$ with $T>0$.  

The boundaries $\Gamma_t$ of $\Omega_t$ can be written as graphs over $\Gamma_0:= \partial \Omega_0$.  In particular, using $x$ as a coordinate, part of $\Gamma_t$ is given by a graph $y=g(x,t)$ for $x_0 \le x \le x_1$, with $g(x,0)=0$ for $x_0 \le x \le x_1$, with the unit normal to $\Omega_0$ being in the negative $y$ direction.

  We may assume that 
$$g\in C([0,T], C^{2+\alpha}([x_0,x_1])) \cap C^1([0,T], C^{1+\alpha}([x_0,x_1]).$$
Moreover, $(\partial g/\partial t)(x,0)$ represents the \emph{negative} of the velocity in the normal direction at time $t=0$.  Hence
by (iii) of Lemma \ref{initialdata}, 
$$\frac{1}{2} \left(  \frac{\partial g}{\partial t} (x_0, 0) + \frac{\partial g}{\partial t}(x_1,0) \right) <  \frac{\partial g}{\partial t}\left(\frac{x_0+x_1}{2},0 \right).$$
Then for $t \in (0,\delta]$ for $\delta>0$ sufficiently small, we have
$$\frac{1}{2} \left(  g (x_0, t) + g(x_1,t) \right) <  g\left( \frac{x_0+x_1}{2},t \right).$$
In particular, $x \mapsto g(x,t)$ is not convex for $(x,t) \in [x_0, x_1] \times (0,\delta]$.  Hence $\Omega_t$ is not a convex domain for $t \in (0,\delta]$.
\end{proof}


\begin{thebibliography}{0}
\bibitem{C} Cox, R.G. {\em The dynamics of the spreading of liquids on a solid surface.  Part 1. Viscous Flow.}, J. Fluid Mech. 168 (1986), 169--194
	\bibitem{EG}  Escher, J., Guidotti, P. {\em Local well-posedness for a quasi-stationary droplet model}, Calc. Var. Partial Differential Equations (2015) 54:1147–1160 
\bibitem{FK} Feldman, W.M., Kim, I.C. {\em Dynamic stability of equilibrium capillary drops}, Arch. Ration. Mech. Anal. 211 (2014), no. 3, 819--878
\bibitem{GT} Gilbarg, D., Trudinger, N.S., {\em Elliptic partial differential equations of second order}, Reprint of the 1998 edition. Classics in Mathematics. Springer-Verlag, Berlin, 2001.
\bibitem{Gl} Glasner, K.B. \emph{A boundary integral formulation of quasi-steady fluid wetting}, J. Comput. Phys. 207 (2005), no. 2, 529--541 
\bibitem{GlK} Glasner, K.B., Kim, I.C. {\em Viscosity solutions for a model of contact line motion},  Interfaces Free Bound. 11 (2009), no. 1, 37--60
\bibitem{Gr} Greenspan, H.P. {\em On the motion of a small viscous droplet that wets a surface}, J. Fluid Mech. 84 (1978), 125--143
\bibitem{GK} Grunewald, N., Kim, I.C. {\em A variational approach to a quasi-static droplet model}, 
Calc. Var. Partial Differential Equations 41 (2011), no. 1-2, 1--19
\bibitem{KM} Keady, G., McNabb, A. {\em The elastic torsion problem: solutions in convex domains}, New Zealand J. Math. 22 (1993), no. 2, 43--64
\bibitem{T} Tanner, L. {\em The spreading of silicone oil drops on horizontal surfaces}., J. Phys. D 12 (1979), 1473--1484
\bibitem{V} Voinov, O.V. {\em Hydrodynamics of wetting}, Fluid Dyn. 11 (1976), 714--721
\end{thebibliography}
\end{document}